\documentclass[oneside,english]{amsart}
\usepackage[]{fontenc}
\usepackage[latin9]{inputenc}
\usepackage{color}
\usepackage{babel}
\usepackage{float}
\usepackage{amsthm}
\usepackage{amstext}
\usepackage{amssymb}
\usepackage[numbers]{natbib}
\usepackage[unicode=true,pdfusetitle,
 bookmarks=true,bookmarksnumbered=false,bookmarksopen=false,
 breaklinks=false,pdfborder={0 0 1},backref=false,colorlinks=true]
 {hyperref}
\hypersetup{
 allcolors=blue}
\usepackage{breakurl}

\makeatletter

\floatstyle{ruled}
\newfloat{algorithm}{tbp}{loa}
\providecommand{\algorithmname}{Algorithm}
\floatname{algorithm}{\protect\algorithmname}

\numberwithin{equation}{section}
\numberwithin{figure}{section}
\theoremstyle{plain}
\newtheorem{thm}{\protect\theoremname}
  \theoremstyle{plain}
  \newtheorem{lem}[thm]{\protect\lemmaname}
  \theoremstyle{plain}
  \newtheorem{prop}[thm]{\protect\propositionname}
  \theoremstyle{remark}
  \newtheorem*{rem*}{\protect\remarkname}

\usepackage[a4paper]{geometry} 
\let\tempone\itemize
\let\temptwo\enditemize
\renewenvironment{itemize}{\tempone\addtolength{\itemsep}{0.3\baselineskip}}{\temptwo}

\usepackage{listings}
\usepackage{xcolor}
\makeatother

\usepackage{listings}
  \providecommand{\lemmaname}{Lemma}
  \providecommand{\propositionname}{Proposition}
  \providecommand{\remarkname}{Remark}
\providecommand{\theoremname}{Theorem}

\begin{document}

\title{Permutation polynomials of degree $8$ over finite fields of characteristic
$2$}

\author{Xiang Fan}

\begin{abstract}
Up to linear transformations, we obtain a classification of permutation
polynomials (PPs) of degree $8$ over $\mathbb{F}_{2^{r}}$ with $r>3$.
By {[}J. Number Theory 176 (2017) 46\textendash 66{]}, a polynomial
$f$ of degree $8$ over $\mathbb{F}_{2^{r}}$ is exceptional if and
only if $f-f(0)$ is a linearized PP. So it suffices to search for
non-exceptional PPs of degree $8$ over $\mathbb{F}_{2^{r}}$, which
exist only when $r\leqslant9$ by a previous result. This can be exhausted
by the SageMath software running on a personal computer. To facilitate
the computation, some requirements after linear transformations and
explicit equations by Hermite\textquoteright s criterion are provided
for the polynomial coefficients. The main result is that a non-exceptional
PP $f$ of degree $8$ over $\mathbb{F}_{2^{r}}$ (with $r>3$) exists
if and only if $r\in\{4,5,6\}$, and such $f$ is explicitly listed
up to linear transformations.
\end{abstract}



\maketitle

\section{Introduction}

Let $\mathbb{F}_{q}$ denote the finite field of characteristic $p$
and order $q=p^{r}$, and let $\mathbb{F}_{q}^{*}=\mathbb{F}_{q}\backslash\{0\}$.
Reserve the letter $x$ for the indeterminate of the polynomial ring
$\mathbb{F}_{q}[x]$ with coefficients in $\mathbb{F}_{q}$. We call
$f\in\mathbb{F}_{q}[x]$ a \emph{permutation polynomial} (PP) over
$\mathbb{F}_{q}$ if the induced map $a\mapsto f(a)$ permutes $\mathbb{F}_{q}$.
Initiated by Hermite \citep{Hermite1863sur}  and Dickson \citep{Dickson1897analytic}
in the 19th century, the study of PPs has drawn much attention, with
more and more classes of PPs (with either nice appearance or certain
desired properties) found or constructed. Some classes of them have
significant applications in wide areas of mathematics and engineering
such as cryptography, coding theory, combinatorial designs. However,
the non-trivial problems of classification of PPs (of certain prescribed
forms) are still challenging.

Especially, the classification of PPs of a given degree $d$ over
an arbitrary $\mathbb{F}_{q}$ is complete known only for $d\leqslant7$:
\begin{itemize}
\item by Dickson\textquoteright s 1896 thesis \citep{Dickson1897analytic}
for $d\leqslant5$ with any $q$, and for $d=6$ with any odd $q$;
\item by Li, Chandler and Xiang \citep{LiChandlerXiang2010permutation}
in 2010 for $d=6$ or $7$ with $q=2^{r}$ (for any $r\geqslant3$);
\item by the author's recent work \citep{Fan2019PP7} for $d=7$ with any
odd prime power $q$.
\end{itemize}
The present paper aims for a classification of PPs of degree $8$
over $\mathbb{F}_{2^{r}}$ (with any integer $r>3$) up to linear
transformations.

Similar to \citep{Fan2019PP7}, our approach here is based on some
known results on exceptional polynomials. In this paper, an \emph{exceptional
polynomial} over $\mathbb{F}_{q}$ is defined as a PP over $\mathbb{F}_{q}$
which is also a PP over $\mathbb{F}_{q^{m}}$ for infinitely many
integers $m\geqslant1$. Recall that $\mathbb{F}_{q}$ is of characteristic
$p$. A polynomial $\varphi$ in $\mathbb{F}_{q}[x]$ of the form
$\varphi(x)=\sum_{s=0}^{t}c_{s}x^{p^{s}}$ (with all $c_{s}\in\mathbb{F}_{q}$)
is called a \emph{linearized polynomial}. Clearly, a linearized polynomial
$\varphi$ induces a $\mathbb{F}_{p}$-linear map from $\mathbb{F}_{q}$
to itself, so $\varphi$ is a PP over $\mathbb{F}_{q}$ if and only
if $\mathbb{F}_{q}^{*}$ contains no root of $\varphi$. On the one
hand, a linearized PP $\varphi$ over $\mathbb{F}_{q}$ is always
exceptional. Indeed, let $\mathbb{F}_{q^{k}}$ be the splitting field
of $\varphi$ over $\mathbb{F}_{q}$, then every $\mathbb{F}_{q^{m}}^{*}$
with $\mathrm{gcd}(m,k)=1$ contains no root of $\varphi$. On the
other hand, the recent work of Bartoli, Giulietti, Quoos and Zini
\citep[Proposition 7.1]{BartoliGiuliettiQuoosZini2017complete} showed
that a polynomial $f\in\mathbb{F}_{2^{r}}[x]$ (with $r>3$) of degree
$8$ is exceptional if and only if $f(x)-f(0)$ is linearized and
has no root in $\mathbb{F}_{2^{r}}^{*}$. Also note an explicit criterion
\citep[\S\S58]{Dickson1897analytic} for a linearized polynomial over
$\mathbb{F}_{q}$ to have no root in $\mathbb{F}_{q}^{*}$. Combining
them, we get the following explicit  determination of exceptional
polynomials of degree $8$ over $\mathbb{F}_{2^{r}}$ (with $r>3$).
\begin{lem}
[{\citep[Proposition 7.1]{BartoliGiuliettiQuoosZini2017complete}
and \citep[\S\S58]{Dickson1897analytic}}]\label{lem:EP8} A polynomial
$f(x)=\sum_{i=0}^{8}a_{i}x^{i}\in\mathbb{F}_{2^{r}}[x]$ (with $r>3$,
all $a_{i}\in\mathbb{F}_{2^{r}}$ and $a_{8}\ne0$) is exceptional
over $\mathbb{F}_{2^{r}}$ if and only if $a_{7}=a_{6}=a_{5}=a_{3}=0$
(i.e. $f(x)-a_{0}$ is linearized) and $\det(c_{i-j}^{p^{j}})_{0\leqslant i,j\leqslant r-1}\neq0$,
where 
\[
(c_{0},c_{1},c_{2},c_{3})=(c_{-r},c_{-r+1},c_{-r+2},c_{-r+3})=(a_{1},a_{2},a_{4},a_{8})
\]
and all other $c_{s}=0$.
\end{lem}
Thereafter, to complete the classification of PPs of degree $8$ over
$\mathbb{F}_{2^{r}}$ (with $r>3$), it suffices to search for the
non-exceptional ones. It is well-known that a non-exceptional PP of
degree $n$ exists over $\mathbb{F}_{q}$ only if $q\leqslant C_{n}$,
where $C_{n}$ stands for a constant depending only on $n$. The proof
of this fact can be found in \citep{DavenportLewis1963notes,Hayes1967geometric,Wan1987conjecture}
for abstract $C_{n}$, in \citep{Gathen1991values} for $C_{n}=n^{4}$,
in \citep{ChahalGhorpade2018Carlitz} for a bound less than $n^{2}(n-2)^{2}$,
and in \citep{Fan2019Weil} for 
\[
C_{n}=\left\lfloor \left(\dfrac{(n-2)(n-3)+\sqrt{(n-2)^{2}(n-3)^{2}+8n-12}}{2}\right)^{2}\right\rfloor ,
\]
where $\lfloor t\rfloor$ denotes the greatest integer not exceeding
a real number $t$. When $n=8$, the last bound is $925$, which indicates
the following lemma.
\begin{lem}
\label{lem:bound} A non-exceptional PP of degree $8$ exists over
$\mathbb{F}_{2^{r}}$ only if $r\leqslant9$.
\end{lem}
In this paper, we search for non-exceptional PPs of degree $8$ over
$\mathbb{F}_{2^{r}}$ (with $4\leqslant r\leqslant9$), with the help
of \href{http://www.sagemath.org}{SageMath} \citep{sagemath}, an
open-source computer algebra system with features covering many aspects
of mathematics, including algebra, calculus, combinatorics, graph
theory, number theory, numerical analysis and statistics. SageMath
uses a syntax resembling Python\textquoteright s, which is easy to
understand for readers without prior programming experience. We run
all algorithms in this paper on the version 8.6 of SageMath.

Most of our efforts are devoted to reduce the number of searching
candidates for non-exceptional PPs. The structure of this paper is
as follows. Section \ref{sec:trans} investigates linear transformations
of polynomials of degree $8$ over $\mathbb{F}_{2^{r}}$, and imposes
some constraints on the polynomial coefficients after linear transformations.
Section \ref{sec:HC} establishes Algorithm \ref{alg:HC} for explicit
equations on coefficients of PPs over $\mathbb{F}_{2^{r}}$, by Hermite\textquoteright s
criterion and a multinomial analogue of the Lucas theorem. Combining
constraints from Section \ref{sec:trans} and equations as outputs
of Algorithm \ref{alg:HC}, we analysis the polynomial coefficients
of non-exceptional PPs of degree $8$ in Section \ref{sec:Case},
on a case-by-case basis for $4\leqslant r\leqslant9$. It is verified
that all PPs of degree $8$ over $\mathbb{F}_{2^{r}}$ with $r\in\{7,8,9\}$
are exceptional. For $r\in\{4,5,6\}$, we write explicit SageMath
codes to test all remaining candidates for non-exceptional PPs of
degree $8$ over $\mathbb{F}_{2^{r}}$. We also rewrite the SageMath
outputs as Theorem \ref{thm:r=00003D4}, \ref{thm:r=00003D5}, and
\ref{thm:r=00003D6}, listing all non-exceptional PPs of degree $8$
over $\mathbb{F}_{2^{r}}$, up to linear transformations.

\section{\label{sec:trans}Linear Transformations}

To reduce the number of candidates in our search for non-exceptional
PPs of degree $8$ over $\mathbb{F}_{2^{r}}$, we consider the classification
up to linear transformations. We say that two polynomials $f$ and
$g$ in $\mathbb{F}_{q}[x]$ are \emph{related by linear transformations}
(\emph{linearly related} for short) if there exist $s,t\in\mathbb{F}_{q}^{*}$
and $u,v\in\mathbb{F}_{q}$ such that $g(x)=sf(tx+u)+v$. Linearly
related $f$ and $g$ hold the same degree, and $f$ is a (non-exceptional)
PP over $\mathbb{F}_{q}$ if and only if so is $g$.

Each PP of degree $8$ over $\mathbb{F}_{2^{r}}$ is linearly ralated
to some $f\in\mathbb{F}_{2^{r}}[x]$ in normalized form, i.e. $f(x)=x^{8}+\sum_{i=1}^{7}a_{i}x^{i}$
with all $a_{i}\in\mathbb{F}_{2^{r}}$. As we mentioned before, a
linearized PP must be exceptional. So $(a_{7},a_{6},a_{5},a_{3})\neq(0,0,0,0)$
for $f$ to be non-exceptional. Furthermore, a case-by-case analysis
in Section \ref{sec:Case} by Hermite's criterion will show that $(a_{7},a_{6},a_{5})\neq(0,0,0)$
if $f$ is a non-exceptional PP over $\mathbb{F}_{2^{r}}$ with $4\leqslant r\leqslant9$.
For later use, up to linear transformations, more constraints on the
coefficients $a_{i}$'s can be imposed by the following Proposition
\ref{prop:Req}.
\begin{prop}
\label{prop:Req} Let $e$ be a generator of the multiplicative group
$\mathbb{F}_{2^{r}}^{*}$. For each $a\in\mathbb{F}_{2^{r}}^{*}$,
fix an element $\omega(a)$ in the set $\mathbb{F}_{2^{r}}\backslash\{u^{2}+au:u\in\mathbb{F}_{2^{r}}\}$.
Then each polynomial of degree $8$ over $\mathbb{F}_{2^{r}}$ is
linearly ralated to some $f(x)=x^{8}+\sum_{i=1}^{7}a_{i}x^{i}\in\mathbb{F}_{2^{r}}[x]$
with all $a_{i}\in\mathbb{F}_{2^{r}}$ satisfying the following requirements
$(\mathbf{R1})\sim(\mathbf{R3})$:

$(\mathbf{R1})$ $(a_{7},a_{6})\in\{(1,0),(0,1),(0,0)\}$.

$(\mathbf{R2})$ if $(a_{7},a_{6})=(0,1)$, then $a_{4}\in\begin{cases}
\{0\} & \text{if }a_{5}=0,\\
\{0,\omega(a_{5})\} & \text{if }a_{5}\neq0.
\end{cases}$

$(\mathbf{R3})$ if $a_{7}=a_{6}=0\neq a_{5}$, then $a_{4}=0$ and
$a_{5}\in\begin{cases}
\{1\} & \text{if }r\text{ is odd},\\
\{1,e,e^{2}\} & \text{if }r\text{ is even}.
\end{cases}$

Moreover, suppose $(a_{7},a_{6},a_{5})\neq(0,0,0)$ and let $g(x)=x^{8}+\sum_{i=1}^{7}a_{i}'x^{i}\in\mathbb{F}_{2^{r}}[x]$
with all $a_{i}'\in\mathbb{F}_{2^{r}}$ satisfying the same requirements
$(\mathbf{R1})\sim(\mathbf{R3})$, then $f$ and $g$ are linearly
related if and only if one of the following happens: 

$(i)$ $g=f$.

$(ii)$ $g(x)=f(x+a_{5})-f(a_{5})$ with $(a_{7},a_{6})=(0,1)$ and
$a_{5}(a_{3}+a_{5}^{3})\neq0$. In this case, 
\[
(a_{7}',a_{6}',a_{5}',a_{4}',a_{3}',a_{2}',a_{1}')=(a_{7},a_{6},a_{5},a_{4},a_{3},a_{2}+a_{3}a_{5}+a_{5}^{4},a_{1}+a_{3}a_{5}^{2}+a_{5}^{5}).
\]

$(iii)$ $r$ is even, $a_{7}=a_{6}=a_{4}=0\neq a_{5}$ and $g(x)=t^{-8}f(tx)$
with $t\in\{e^{\frac{2^{r}-1}{3}},e^{\frac{2(2^{r}-1)}{3}}\}$.\end{prop}
\begin{proof}
By definition, each polynomial of degree $8$ over $\mathbb{F}_{2^{r}}$
is linearly ralated to some $h(x)=x^{8}+\sum_{i=1}^{7}c_{i}x^{i}\in\mathbb{F}_{2^{r}}[x]$
with all $c_{i}\in\mathbb{F}_{2^{r}}$. If $c_{7}\neq0$, let $f(x)=c_{7}^{-8}h(c_{7}x+c_{7}^{-1}c_{6})-c_{7}^{-8}h(c_{7}^{-1}c_{6})=x^{8}+x^{7}+\sum_{i=1}^{5}a_{i}x^{i}$.
If $c_{7}=0\neq c_{6}$, let $f(x)=c_{6}^{-4}h(c_{6}^{2^{r-1}}x)=x^{8}+x^{6}+\sum_{i=1}^{5}a_{i}x^{i}$.
If $c_{7}=c_{6}=0$, let $f=h$. In any case, $h$ is linearly related
to $f(x)=x^{8}+\sum_{i=1}^{7}a_{i}x^{i}$ with all $a_{i}\in\mathbb{F}_{2^{r}}$
satisfying $(\mathbf{R1})$. Henceforth, we only need to adjust $f$
up to linear transformations, to meet $(\mathbf{R2})$ and $(\mathbf{R3})$.

$(\mathbf{R2})$: Suppose $(a_{7},a_{6})=(0,1)$. Then $f(x)=x^{8}+x^{6}+\sum_{i=1}^{5}a_{i}x^{i}$
can be replaced by 
\[
f(x+u)-f(u)=x^{8}+x^{6}+a_{5}x^{5}+(u^{2}+a_{5}u+a_{4})x^{4}+\text{lower terms},
\]
with an arbitrary $u\in\mathbb{F}_{2^{r}}$. If $a_{5}=0$, take $u=a_{4}^{2^{r-1}}$,
then $f(x+u)-f(u)$ satisfies $(\mathbf{R2})$. Hereafter assume $a_{5}\neq0$.
Note that $u\mapsto u^{2}+a_{5}u$ gives a $\mathbb{F}_{2}$-linear
map from $\mathbb{F}_{2^{r}}$ to itself, so $\{u^{2}+a_{5}u:u\in\mathbb{F}_{2^{r}}\}$
is a $\mathbb{F}_{2}$-linear subspace containing exactly one half
of the elements of $\mathbb{F}_{2^{r}}$. Thus $\{u^{2}+a_{5}u:u\in\mathbb{F}_{2^{r}}\}$
contains either $a_{4}$ or $a_{4}+\omega(a_{5})$ (but not both).
Take $u\in\mathbb{F}_{2^{r}}$ such that $u^{2}+a_{5}u\in\{a_{4},a_{4}+\omega(a_{5})\}$,
then $f(x+u)-f(u)$ satisfies $(\mathbf{R2})$.

$(\mathbf{R3})$: Suppose $a_{7}=a_{6}=0\neq a_{5}$. If $a_{4}\neq0$,
we can replace $f(x)=x^{8}+\sum_{i=1}^{5}a_{i}x^{i}$ by 
\[
f(x+u)-f(u)=x^{8}+a_{5}x^{5}+(a_{5}u+a_{4})x^{4}+\text{lower terms},
\]
with $u=a_{5}^{-1}a_{4}$ annihilating the coefficient of $x^{4}$.
Hereafter assume $a_{4}=0$, and replace $f$ by 
\[
t^{-8}f(tx)=x^{8}+t^{-3}a_{5}x^{5}+\sum_{i=1}^{3}t^{i-8}a_{i}x^{i},
\]
with an arbitrary $t\in\mathbb{F}_{2^{r}}^{*}$. Let 
\[
\Lambda=\{e^{i}:0\leqslant i<\mathrm{gcd}(3,2^{r}-1)\}=\begin{cases}
\{1\} & \text{if }r\text{ is odd},\\
\{1,e,e^{2}\} & \text{if }r\text{ is even},
\end{cases}
\]
which is a complete set of coset representatives of $\mathbb{F}_{2^{r}}^{*}/\{t^{3}:t\in\mathbb{F}_{2^{r}}^{*}\}$.
Certain $t\in\mathbb{F}_{2^{r}}^{*}$ ensures that $t^{-3}a_{5}\in\Lambda$,
and thus $t^{-8}f(tx)$ satisfies $(\mathbf{R3})$.

So far we have showed that each polynomial of degree $8$ over $\mathbb{F}_{2^{r}}$
is linearly ralated to some $f(x)=x^{8}+\sum_{i=1}^{7}a_{i}x^{i}\in\mathbb{F}_{2^{r}}[x]$
with all $a_{i}\in\mathbb{F}_{2^{r}}$ satisfying $(\mathbf{R1})\sim(\mathbf{R3})$.
In the following, suppose $(a_{7},a_{6},a_{5})\neq(0,0,0)$ and let
$f(x)$ be linearly related to $g(x)=x^{8}+\sum_{i=1}^{7}a_{i}'x^{i}\in\mathbb{F}_{2^{r}}[x]$
with all $a_{i}'\in\mathbb{F}_{2^{r}}$ satisfying $(\mathbf{R1})\sim(\mathbf{R3})$.
By definition, there exist $s,t\in\mathbb{F}_{2^{r}}^{*}$ and $u,v\in\mathbb{F}_{2^{r}}$
such that 
\begin{align*}
x^{8}+\sum_{i=1}^{7}a_{i}'x^{i} & =g(x)=sf(tx+u)+v\\
 & =st^{8}x^{8}+st^{7}a_{7}x^{7}+s(a_{7}u+a_{6})t^{6}x^{6}+\text{lower terms}.
\end{align*}
Comparing the coefficients of $x^{8}$, we see that $s=t^{-8}$.

(1) Suppose $(a_{7},a_{6})=(1,0)$. As $a_{7}'=st^{7}a_{7}=t^{-1}\neq0$,
$(a_{7}',a_{6}')=(1,0)$ by $(\mathbf{R1})$. So $s=t=1$. Comparing
the coefficients of $x^{6}$, $0=a_{6}'=a_{7}u+a_{6}=u$. So $v=g(0)-f(u)=0$,
and $f=g$.

(2) Suppose $(a_{7},a_{6})=(0,1)$. Note that 
\begin{align*}
x^{8}+\sum_{i=1}^{7}a_{i}'x^{i} & =g(x)=t^{-8}f(tx+u)+v\\
 & =x^{8}+t^{-2}x^{6}+t^{-3}a_{5}x^{5}+t^{-4}(u^{2}+a_{5}u+a_{4})x^{4}+\text{lower terms},
\end{align*}
so $a_{7}'=0\neq a_{6}'=t^{-2}$. By $(\mathbf{R1})$, $(a_{7}',a_{6}')=(0,1)$
and $t^{2}=1$. So $t=1$, $a_{5}'=a_{5}$ and $a_{4}'-a_{4}=u^{2}+a_{5}u$.
By $(\mathbf{R2})$, $a_{4}'=a_{4}\in\{0,\omega(a_{5})\}$, and $u\in\{u\in\mathbb{F}_{2^{r}}:u^{2}+a_{5}u=0\}=\{0,a_{5}\}$.
If $u=0$, then $v=g(0)-f(u)=0$ and $f=g$. Hereafter assume $u=a_{5}$.
Note that 
\begin{align*}
 & x^{8}+x^{6}+\sum_{i=1}^{5}a_{i}'x^{i}=g(x)=f(x+a_{5})-f(a_{5})\\
={} & x^{8}+x^{6}+a_{5}x^{5}+a_{4}x^{4}+a_{3}x^{3}+(a_{2}+a_{3}a_{5}+a_{5}^{4})x^{2}+(a_{1}+a_{3}a_{5}^{2}+a_{5}^{5})x.
\end{align*}
So $(a_{7}',a_{6}',a_{5}',a_{4}',a_{3}',a_{2}',a_{1}')=(a_{7},a_{6},a_{5},a_{4},a_{3},a_{2}+a_{3}a_{5}+a_{5}^{4},a_{1}+a_{3}a_{5}^{2}+a_{5}^{5})$.
In this case, $f\neq g$ if and only if $a_{5}(a_{3}+a_{5}^{3})\neq0$.

(3) Suppose $a_{7}=a_{6}=0\neq a_{5}$. Then $a_{4}=0$ by $(\mathbf{R3})$.
Note that 
\[
x^{8}+\sum_{i=1}^{7}a_{i}'x^{i}=t^{-8}f(tx+u)+v=x^{8}+t^{-3}a_{5}x^{5}+t^{-4}a_{5}ux^{4}+t^{-5}x^{3}+\text{lower terms}.
\]
So $a_{7}'=a_{6}'=0$, and $a_{5}'=t^{-3}a_{5}\neq0$. By $(\mathbf{R3})$,
$a_{5}=a_{5}'=t^{-3}a_{5}$ and $0=a_{4}'=t^{-4}a_{5}u$. Thus $t^{3}=1$
and $u=0$. Therefore, $g(x)=t^{-8}f(tx)$ with some element $t$
in the set 
\[
\{t\in\mathbb{F}_{2^{r}}^{*}:t^{3}=1\}=\begin{cases}
\{1\} & \text{if }r\text{ is odd},\\
\{1,e^{\frac{2^{r}-1}{3}},e^{\frac{2(2^{r}-1)}{3}}\} & \text{if }r\text{ is even}.
\end{cases}
\]
In this case, $f\neq g$ only when $r$ is even and $t\in\{e^{\frac{2^{r}-1}{3}},e^{\frac{2(2^{r}-1)}{3}}\}$.
\end{proof}

\section{\label{sec:HC}Hermite\textquoteright s Criterion}

Dickson \citep[\S11]{Dickson1897analytic} provided a criterion for
PPs over $\mathbb{F}_{q}$ on their coefficients. It is usually called
\emph{Hermite\textquoteright s criterion} as Dickson attributed its
prime field case to Hermite. Assuming some notations, we quote an
explicit version of it from \citep{LidlNiederreiter1997book}.

Let $\mathbb{N}=\{m\in\mathbb{Z}:m\geqslant0\}$. For $m\in\mathbb{N}$
and $f\in\mathbb{F}_{q}[x]$, let $[x^{m}:f]$ denote the coefficient
of $x^{m}$ in $f(x)$. In other words, for a nonzero polynomial $f\in\mathbb{F}_{q}[x]$,
we have $f(x)=\sum_{m=0}^{\deg(f)}[x^{m}:f]\cdot x^{m}$ and $\deg(f)=\max\{m\in\mathbb{N}:[x^{m}:f]\neq0\}$.
\begin{lem}
[{Hermite\textquoteright s criterion \citep[Theorem 7.6]{LidlNiederreiter1997book}}]\label{lem:HC}
A necessary and sufficient condition for $f\in\mathbb{F}_{q}[x]$
to be a PP over $\mathbb{F}_{q}$ is  
\[
\sum_{t=1}^{\lfloor\frac{\deg(f^{k})}{q-1}\rfloor}[x^{t(q-1)}:f^{k}]\begin{cases}
=0 & \text{for }1\leqslant k\leqslant q-2\ (\text{and }\gcd(k,q)=1),\\
\neq0 & \text{for }k=q-1.
\end{cases}
\]

\end{lem}
Recall how to calculate $[x^{m}:f^{k}]$ with the help of multinomial
coefficients. For integers $k$, $j_{1}$, $j_{2}$, $\dots$, and
$j_{t}$, define the \emph{multinomial coefficient} as 
\[
\binom{k}{j_{1},j_{2},\dots,j_{t}}:=\begin{cases}
\dfrac{k!}{j_{1}!j_{2}!\cdots j_{t}!} & \text{if }k=j_{1}+j_{2}+\cdots+j_{t}\text{ and all }j_{1},\dots,j_{t}\geqslant0,\\
0 & \text{otherwise}.
\end{cases}
\]
For $f(x)=x^{8}+\sum_{i=1}^{7}a_{i}x^{i}\in\mathbb{F}_{2^{r}}[x]$
with all $a_{i}\in\mathbb{F}_{2^{r}}$, by the multinomial theorem,
\[
f(x)^{k}=\sum_{j_{1}+j_{2}+\cdots+j_{8}=k}\binom{k}{j_{1},j_{2},\dots,j_{8}}(\prod_{i=1}^{7}a_{i}^{j_{i}})\cdot x^{j_{1}+2j_{2}+3j_{3}+4j_{4}+5j_{5}+6j_{6}+7j_{7}+8j_{8}}.
\]
Let $J(k,m)$ denote the set of all solutions $(j_{1},j_{2},\dots,j_{8})\in\mathbb{N}^{8}$
of the linear equations: 
\[
\left\{ \begin{aligned}j_{1}+j_{2}+j_{3}+j_{4}+j_{5}+j_{6}+j_{7}+j_{8} & =k,\\
j_{1}+2j_{2}+3j_{3}+4j_{4}+5j_{5}+6j_{6}+7j_{7}+8j_{8} & =m.
\end{aligned}
\right.
\]
Then $[x^{m}:f^{k}]={\displaystyle \sum_{(j_{1},\dots,j_{8})\in J(k,m)}\binom{k}{j_{1},j_{2},\dots,j_{8}}a_{1}^{j_{1}}a_{2}^{j_{2}}a_{3}^{j_{3}}a_{4}^{j_{4}}a_{5}^{j_{5}}a_{6}^{j_{6}}a_{7}^{j_{7}}}$.

Working over $\mathbb{F}_{2^{r}}$, the calculation of multinomial
coefficients can be greatly simplified by the following multinomial
analogue (and corollary) of \emph{Lucas's theorem} \citep[(137)]{Lucas1878theorie}.
\begin{lem}
\label{lem:Lucas} \textup{\citep[\S\S14]{Dickson1897analytic}} Let
$p$ be a prime, $k=\sum_{s=0}^{l}k_{s}p^{s}$, and $j_{i}=\sum_{s=0}^{l}j_{i,s}p^{s}$
for $1\leqslant i\leqslant t$, with all $k_{s}$ and $j_{i,s}\in\{0,1,2,\dots,p-1\}$.
Then 
\[
\binom{k}{j_{1},j_{2},\dots,j_{t}}\equiv\prod_{s=0}^{l}\binom{k_{s}}{j_{1,s},j_{2,s},\dots,j_{t,s}}\ (\mathrm{mod}\ p).
\]
In particular,  $p\nmid{\displaystyle \binom{k}{j_{1},j_{2},\dots,j_{t}}}$
if and only if $k_{s}=j_{1,s}+j_{2,s}+\cdots+j_{t,s}$ for any $0\leqslant s\leqslant l$.
\end{lem}
For $n=\sum_{s=0}^{\lfloor\log_{2}(n)\rfloor}n_{s}2^{s}\in\mathbb{N}$
with all $n_{s}\in\{0,1\}$, let 
\begin{align*}
\beta(n) & =\{s\in\mathbb{N}:n_{s}=1,\ 0\leqslant s\leqslant\log_{2}(n)\}\\
 & =\{s\in\mathbb{N}:\lfloor n/2^{s}\rfloor\text{ is odd}\}.
\end{align*}
When $p=2$ in Lemma \ref{lem:Lucas}, note that $\binom{k}{j_{1},j_{2},\dots,j_{t}}\equiv1\ (\mathrm{mod}\ 2)$
if and only if $\beta(k)$ is the disjoint union $\bigsqcup_{i=1}^{t}\beta(j_{i})$
of its subsets $\beta(j_{1}),\beta(j_{2})\dots,\beta(j_{t})$ (i.e.
$\beta(k)=\bigcup_{i=1}^{t}\beta(j_{i})$ and $\beta(j_{i})\cap\beta(j_{i'})=\O$
for any $i\neq i'$). Therefore, 
\[
[x^{m}:f^{k}]=\sum_{\substack{\beta(k)=\bigsqcup_{i=1}^{8}\beta(j_{i})\\
j_{1}+2j_{2}+3j_{3}+\cdots+8j_{8}=m
}
}a_{1}^{j_{1}}a_{2}^{j_{2}}a_{3}^{j_{3}}a_{4}^{j_{4}}a_{5}^{j_{5}}a_{6}^{j_{6}}a_{7}^{j_{7}}\in\mathbb{F}_{2^{r}}.
\]
Suppose $\beta(k)=\{s_{0},s_{1},\dots,s_{n-1}\}$, then the set $\{(j_{1},j_{2},\dots,j_{8})\in\mathbb{N}^{8}:\beta(k)=\bigsqcup_{i=1}^{8}\beta(j_{i})\}$
contains exactly $8^{n}$ elements, which can be listed explicitly
by the one-to-one correspondence: 
\begin{eqnarray*}
\{(j_{1},j_{2},\dots,j_{8})\in\mathbb{N}^{8}:\beta(k)=\bigsqcup_{i=1}^{8}\beta(j_{i})\} & \stackrel{\simeq}{\longleftrightarrow} & \{0,1,2,3,\dots,8^{n}-1\}.\\
(j_{1},j_{2},\dots,j_{8})\text{ with} & \leftrightarrow & u=\sum_{v=0}^{n-1}u_{v}8^{v}\text{ with}\\
j_{i}=\begin{cases}
\sum_{\stackrel{0\leqslant v<n}{u_{v}=i}}2^{s_{v}} & \text{if }1\leqslant i\leqslant7\\
\sum_{\stackrel{0\leqslant v<n}{u_{v}=0}}2^{s_{v}} & \text{if }i=8
\end{cases} &  & u_{v}=\begin{cases}
i & \text{if }s_{v}\in\beta(j_{i})\text{ and }1\leqslant i\leqslant7\\
0 & \text{if }s_{v}\in\beta(j_{8})
\end{cases}
\end{eqnarray*}

By the above arguments, we can calculate the left hand side in Lemma
\ref{lem:HC} (Hermite\textquoteright s criterion) for the $k$-th
power of a PP $f(x)=x^{8}+\sum_{i=1}^{7}a_{i}x^{i}$ over $\mathbb{F}_{q}$
(with $q=2^{r}$) as 
\begin{align*}
 & \mathbf{HC}(r,k,a_{7},a_{6},a_{5},a_{4},a_{3},a_{2},a_{1})=\sum_{t=1}^{\lfloor\frac{8k}{q-1}\rfloor}[x^{t(q-1)}:f^{k}]\\
=\  & \sum_{\substack{0\leqslant u<8^{n}\\
\sum_{i=1}^{7}ij_{i}(u)+8j_{0}(u)\in\{t(q-1):\,1\leqslant t\leqslant\lfloor\frac{8k}{q-1}\rfloor\}
}
}a_{1}^{j_{1}(u)}a_{2}^{j_{2}(u)}a_{3}^{j_{3}(u)}a_{4}^{j_{4}(u)}a_{5}^{j_{5}(u)}a_{6}^{j_{6}(u)}a_{7}^{j_{7}(u)},
\end{align*}
with $j_{i}(u)=\sum_{\stackrel{0\leqslant v<n}{u_{v}=i}}2^{s_{v}}$
for $0\leqslant i\leqslant7$, where $\{s_{0},s_{1},\dots,s_{n-1}\}=\beta(k)$,
and $u=\sum_{v=0}^{n-1}u_{v}8^{v}$ with all $u_{v}\in\{0,1,2,\dots,7\}$.

\begin{algorithm}[h]
\protect\caption{Hermite\textquoteright s criterion for $k$-th power of a PP $f(x)=x^{8}+\sum_{i=1}^{7}a_{i}x^{i}$
over $\mathbb{F}_{2^{r}}$}
\label{alg:HC}

\begin{lstlisting}
def HC(r,k,a7,a6,a5,a4,a3,a2,a1):
    s = [i for i in range(log(k)/log(2)+1) if is_odd(k//2^i)]
    H = 0; n = len(s); q=2^r
    for u in range(8^n):
        j = [0,0,0,0,0,0,0,0]
        for v in range(n):
            j[(u//8^v)%8] += 2^s[v]
        m = j[1]+2*j[2]+3*j[3]+4*j[4]+5*j[5]+6*j[6]+7*j[7]+8*j[0]
        if m in range(q-1,8*k+1,q-1):
            H += a1^j[1]*a2^j[2]*a3^j[3]*a4^j[4]*a5^j[5]*a6^j[6]*a7^j[7]
    return H
\end{lstlisting}
\end{algorithm}

Algorithm \ref{alg:HC} realizes $\mathbf{HC}(r,k,a_{7},a_{6},a_{5},a_{4},a_{3},a_{2},a_{1})$
as a SageMath function. When running it, we can either input some
explicit values of $a_{i}$, or input the indeterminate $a_{i}$ itself
as working over the polynomial ring $\mathbb{F}_{2}[a_{1},a_{2},a_{3},a_{4},a_{5},a_{6},a_{7}]$.
For example, we can calculate $\mathbf{HC}(5,27,a_{7},0,a_{5},1,a_{3},a_{2},a_{1})$
by the following SageMath codes:
\begin{lstlisting}
K.<a1,a2,a3,a4,a5,a6,a7> = PolynomialRing(GF(2))
HC(5,27,a7,0,a5,1,a3,a2,a1)
\end{lstlisting}

By Lemma \ref{lem:HC} (Hermite\textquoteright s criterion), an output
of $\mathbf{HC}$ for specifice $r$ and $k$ (such that $1\leqslant k\leqslant2^{r}-2$)
provides an explicit equation satisfied by the polynomial coefficients
of a PP over $\mathbb{F}_{2^{r}}$.

\section{\label{sec:Case}Case by Case}

In this section, we analysis non-exceptional PPs over $\mathbb{F}_{2^{r}}$
of the form $f(x)=x^{8}+\sum_{i=1}^{7}a_{i}x^{i}$ with all $a_{i}\in\mathbb{F}_{2^{r}}$
subject to requirements  $(\mathbf{R1})\sim(\mathbf{R3})$ of Proposition
\ref{prop:Req}, on a case-by-case basis for $r\in\{4,5,6,7,8,9\}$,
in light of the aforementioned equations 
\[
\mathbf{HC}(r,k,a_{7},a_{6},a_{5},a_{4},a_{3},a_{2},a_{1})=0
\]
with $1\leqslant k\leqslant2^{r}-2$, from outputs of Algorithm \ref{alg:HC}
running on SageMath 8.6.

For later use, we define a SageMath function $\mathbf{PP}(q,e,a_{7},a_{6},a_{5},a_{4},a_{3},a_{2},a_{1})$
in Algorithm \ref{alg:PP} to examine whether $f(x)=x^{8}+\sum_{i=1}^{7}a_{i}x^{i}$
is a PP over $\mathbb{F}_{q}$. The following Lemma \ref{lem:Wan}
ensures that it suffices to check whether the values $f(e^{j})$ for
$0\leqslant j\leqslant\lfloor q-\frac{q-1}{8}\rfloor$ are distinct,
provided a generator $e$ of the multiplicative group $\mathbb{F}_{q}^{*}$.
\begin{lem}
[Wan \citep{Wan1993padic}] \label{lem:Wan} A polynomial $f\in\mathbb{F}_{q}[x]$
of degree $n\geqslant1$ is a PP over $\mathbb{F}_{q}$ if its value
set $\{f(c):c\in\mathbb{F}_{q}\}$ contains at least $\lfloor q-\frac{q-1}{n}\rfloor+1$
distinct values.
\end{lem}
\begin{algorithm}[h]
\protect\caption{$\mathbf{PP}$ to examine whether $f(x)=x^{8}+\sum_{i=1}^{7}a_{i}x^{i}$
is a PP over $\mathbb{F}_{q}$}
\label{alg:PP}

\begin{lstlisting}
def PP(q,e,a7,a6,a5,a4,a3,a2,a1):
    V = []
    for j in range(1+int(q-(q-1)/8)):
        x = e^j
        v = x^8+a7*x^7+a6*x^6+a5*x^5+a4*x^4+a3*x^3+a2*x^2+a1*x
        if v in V: return False
        else: V.append(v)
    return True
\end{lstlisting}
\end{algorithm}

\begin{itemize}
\item Note that $(a_{7},a_{6},a_{5},a_{3})\neq(0,0,0,0)$ for $f$ to be
a non-exceptional PP over $\mathbb{F}_{2^{r}}$. 
\item Let $e$ be a generator of the multiplicative group $\mathbb{F}_{2^{r}}^{*}$.
\item For each $a\in\mathbb{F}_{2^{r}}^{*}$, fix an element $\omega(a)$
in the set $\mathbb{F}_{2^{r}}\backslash\{u^{2}+au:u\in\mathbb{F}_{2^{r}}\}$.
\end{itemize}

\subsection{Case $r=4$, $q=16$}

Note that
\begin{align*}
0 & =\mathbf{HC}(4,3,a_{7},a_{6},a_{5},a_{4},a_{3},a_{2},a_{1})=a_{5}^{3}+a_{3}a_{6}^{2}+a_{4}^{2}a_{7}+a_{1}a_{7}^{2},\\
0 & =\mathbf{HC}(4,5,a_{7},a_{6},a_{5},a_{4},a_{3},a_{2},a_{1})=a_{3}^{5}+a_{6}^{5}+a_{2}^{4}a_{7}+a_{2}a_{7}^{4}.
\end{align*}
If $a_{7}=a_{6}=0$, then $a_{5}=a_{3}=0$. As $(a_{7},a_{6},a_{5},a_{3})\neq(0,0,0,0)$,
we have $(a_{7},a_{6})\neq(0,0)$. By $(\mathbf{R1})$, $(a_{7},a_{6})=(1,0)$
or $(0,1)$.

(1) Suppose $(a_{7},a_{6})=(1,0)$. Then $a_{1}=a_{5}^{3}+a_{4}^{2}$
and $a_{3}^{5}=a_{2}^{4}+a_{2}$.

(2) Suppose $(a_{7},a_{6})=(0,1)$. Then $a_{3}^{5}=1$ and $a_{3}=a_{5}^{3}\neq0$.
By $(\mathbf{R2})$, $a_{4}\in\{0,\omega(a_{5})\}$.

\begin{algorithm}[h]
\protect\caption{\label{alg:r=00003D4} Search for non-exceptional PPs of degree $8$
over $\mathbb{F}_{16}$}

\begin{lstlisting}
q = 2^4; F.<e> = GF(q,modulus=x^4+x+1,repr='log')
for a5 in F: ### a7=1,a6=0 ###
    for (a4,a3) in [(a4,a3) for a4 in F for a3 in F]:
        for a2 in [a2 for a2 in F if a3^5==a2^4+a2]:
            a1 = a5^3 + a4^2
            if PP(q,e,1,0,a5,a4,a3,a2,a1):
                print(F(1),0,a5,a4,a3,a2,a1)
for a5 in [a5 for a5 in F if a5!=0]: ### a7=0,a6=1,a5!=0 ###
    Image = [u^2+a5*u for u in F]
    for w in F:
        if w not in Image: break
    for a4 in [0,w]:
        a3 = a5^3
        for (a2,a1) in [(a2,a1) for a2 in F for a1 in F]:
            if PP(q,e,0,1,a5,a4,a3,a2,a1):
                print(0,F(1),a5,a4,a3,a2,a1)
\end{lstlisting}
\end{algorithm}

Therefore, we write Algorithm \ref{alg:r=00003D4} to search for all
non-exceptional PPs of degree $8$ over $\mathbb{F}_{16}$ up to linear
transformations, and run it in SageMath 8.6.
\begin{itemize}
\item In Algorithm \ref{alg:r=00003D4}, we take $e$ as a root of the Conway
polynomial\footnote{The Conway polynomial $C_{p,n}$ is a particular irreducible polynomial
of degree $n$ over $\mathbb{F}_{p}$ named after John H. Conway by
Richard A. Parker, satisfying a certain compatibility condition proposed
by Conway. The Conway polynomial is chosen to be primitive, so that
each of its roots generates the multiplicative group $\mathbb{F}_{p^{n}}$.
See its \href{https://en.wikipedia.org/wiki/Conway_polynomial_(finite_fields)}{wikipedia page}
and Frank L\"{u}beck's webpage ``\href{http://www.math.rwth-aachen.de/~Frank.Luebeck/data/ConwayPol/index.html}{Conway polynomials for finite fields}
for more information.} $x^{4}+x+1$ in $\mathbb{F}_{16}$, which is a generator of the multiplicative
group $\mathbb{F}_{16}^{*}$. In the outputs, an element $e^{i}$
of $\mathbb{F}_{16}$ is represented by the integer $i$ such that
$1\leqslant i\leqslant q-1$, while the zero element of $\mathbb{F}_{16}$
is represented by $0$.
\item The outputs of Algorithm \ref{alg:r=00003D4} are $113$ tuples of
the form $(0,1,a_{5},a_{4},a_{3},a_{2},a_{1})$ such that $a_{3}=a_{5}^{3}\neq0$.
By Proposition \ref{prop:Req}, they are not linearly related to each
other, and correspond to all $113$ linearly related classes of non-exceptional
PPs of degree $8$ over $\mathbb{F}_{16}$.
\end{itemize}
To save space, we will write down the result up to a composition of
the \emph{Frobenius automorphism} ($\mathbb{F}_{2^{r}}\to\mathbb{F}_{2^{r}}$
with $a\mapsto a^{2}$) with itself, as the following Proposition
\ref{prop:Frobenius} indicates.
\begin{prop}
\label{prop:Frobenius} For $1\leqslant r\in\mathbb{Z}$, let $\Gamma(r)$
be a subset of $\mathbb{F}_{2^{r}}^{*}$ such that 
\[
\mathbb{F}_{2^{r}}^{*}=\{a^{2^{j}}:a\in\Gamma(r),\ 0\leqslant j\leqslant r-1\}.
\]
For each $a\in\Gamma(r)$, fix an element $\omega(a)$ in the set
$\mathbb{F}_{2^{r}}\backslash\{u^{2}+au:u\in\mathbb{F}_{2^{r}}\}$.
Let $h(x)=x^{8}+x^{6}+\sum_{i=1}^{5}c_{i}x^{i}\in\mathbb{F}_{2^{r}}[x]$
with all $c_{i}\in\mathbb{F}_{2^{r}}$ and $c_{5}\neq0$. Then $h(x)$
is linearly related to some $x^{8}+x^{6}+\sum_{i=1}^{5}a_{i}^{2^{j}}x^{i}\in\mathbb{F}_{2^{r}}[x]$
with $0\leqslant j\leqslant r-1$, $a_{5}\in\Gamma(r)$, $a_{4}\in\{0,\omega(a)\}$
and all $a_{i}\in\mathbb{F}_{2^{r}}$. Moreover, $h$ is a (non-exceptional)
PP over $\mathbb{F}_{2^{r}}$ if and only if so is $x^{8}+x^{6}+\sum_{i=1}^{5}a_{i}x^{i}$.\end{prop}
\begin{proof}
By definition, there exist $a_{5}\in\Gamma(r)$ and $0\leqslant j\leqslant r-1$
such that $c_{5}=a_{5}^{2^{j}}$. Let 
\[
\varphi(x)=x^{8}+x^{6}+a_{5}x^{5}+\sum_{i=1}^{4}c_{i}^{2^{r-j}}x^{i}.
\]
By the same arguments for $(\mathbf{R2})$ of Proposition \ref{prop:Req},
$\varphi(x)$ is linearly related to some $f(x)=x^{8}+x^{6}+a_{5}x^{5}+\sum_{i=1}^{4}a_{i}x^{i}$
with all $a_{i}\in\mathbb{F}_{2^{r}}$ and $a_{4}\in\{0,\omega(a)\}$.
Then $h(x)$ is linearly related to $g(x)=x^{8}+x^{6}+a_{5}^{2^{j}}x^{5}+\sum_{i=1}^{4}a_{i}^{2^{j}}x^{i}$.
For $1\leqslant m\in\mathbb{Z}$, as $\mathbb{F}_{2^{rm}}\to\mathbb{F}_{2^{rm}}$
($a\mapsto a^{2^{j}}$) is a field automorphism, $h$ permutes $\mathbb{F}_{2^{rm}}$
$\Leftrightarrow$ $\varphi$ permutes $\mathbb{F}_{2^{rm}}$ $\Leftrightarrow$
$f$ permutes $\mathbb{F}_{2^{rm}}$ $\Leftrightarrow$ $g$ permutes
$\mathbb{F}_{2^{rm}}$.
\end{proof}
Let us take $\Gamma(4)=\{1,e,e^{3},e^{5},e^{7}\}$. By Proposition
\ref{prop:Frobenius}, up to a composition of the Frobenius automorphism
with itself, we only need to pick up outputting tuples $(0,1,a_{5},a_{4},a_{3},a_{2},a_{1})$
such that $a_{5}\in\{1,e,e^{3},e^{5},e^{7}\}$, which gives the following
theorem.
\begin{thm}
\label{thm:r=00003D4}Let $e$ be a root of the Conway polynomial
$x^{4}+x+1$ in $\mathbb{F}_{16}$. Each non-exceptional PP of degree
8 over $\mathbb{F}_{16}$ is linearly related to a polynomial of the
form $x^{8}+x^{6}+\sum_{i=1}^{5}a_{i}^{2^{j}}x^{i}$, with $j\in\{0,1,2,3\}$
and $(a_{5},a_{4},a_{3},a_{2},a_{1})\in\mathbb{F}_{16}^{5}$ in the
following list: 
\begin{align*}
 & (e,0,e^{3},e^{5},e), &  & (e,0,e^{3},e^{9},e^{10}), &  & (e,e,e^{3},e^{8},e^{11}), &  & (e,e,e^{3},e^{9},e^{12}),\\
 & (e,e,e^{3},e^{11},e^{4}), &  & (e,e,e^{3},e^{12},e^{9}), &  & (e,e,e^{3},e^{13},e^{4}), &  & (e^{3},0,e^{9},e^{5},1),\\
 & (e^{3},0,e^{9},e^{8},e^{8}), &  & (e^{3},e^{2},e^{9},0,e^{12}), &  & (e^{3},e^{2},e^{9},e^{7},e^{12}), &  & (e^{3},e^{2},e^{9},1,e^{11}),\\
 & (e^{5},0,1,e^{3},1), &  & (e^{5},0,1,e^{10},e), &  & (e^{5},0,1,e^{10},e^{4}), &  & (e^{5},0,1,e^{12},1),\\
 & (e^{5},e,1,0,e^{13}), &  & (e^{5},e,1,0,e^{14}), &  & (e^{5},e,1,e^{5},e^{3}), &  & (e^{5},e,1,e^{5},e^{13}),\\
 & (e^{5},e,1,e^{7},e^{7}), &  & (e^{5},e,1,e^{9},e^{6}), &  & (e^{5},e,1,e^{9},e^{7}), &  & (e^{7},0,e^{6},e,e^{7}),\\
 & (e^{7},0,e^{6},e^{2},e^{5}), &  & (e^{7},0,e^{6},e^{9},e^{7}), &  & (e^{7},0,e^{6},e^{9},e^{14}), &  & (e^{7},0,e^{6},e^{11},e^{6}),\\
 & (e^{7},0,e^{6},e^{12},e^{2}), &  & (e^{7},e^{2},e^{6},e,e^{9}), &  & (e^{7},e^{2},e^{6},e^{5},e), &  & (e^{7},e^{2},e^{6},e^{12},e^{10}),\\
 & (1,0,1,e,1), &  & (1,e^{3},1,e,e^{13}), &  & (1,e^{3},1,e^{2},e^{13}).
\end{align*}
\end{thm}
\begin{rem*}
Outside the list of Theorem \ref{thm:r=00003D4} are there four other
outputting tuples of Algorithm \ref{alg:r=00003D4} with $a_{5}\in\Gamma(4)=\{1,e,e^{3},e^{5},e^{7}\}$:
\begin{alignat*}{4}
 & (1,0,1,e^{2},1), & \quad & (1,0,1,e^{4},1), & \quad & (1,0,1,e^{8},1), & \quad & (1,e^{3},1,e^{4},e^{13}).
\end{alignat*}
The first three tuples are clearly equivalent to $(1,0,1,e,1)$ up
to a composition of the Frobenius automorphism with itself. The last
tuple gives $g(x)=x^{8}+x^{6}+x^{5}+e^{3}x^{4}+x^{3}+e^{4}x^{2}+e^{13}x$,
linearly related to $g(x+e_{3})+e^{13}=x^{8}+x^{6}+x^{5}+e^{6}x^{4}+x^{3}+e^{2}x^{2}+e^{26}x$,
corresponding to the tuple $(1,e^{6},1,e^{2},e^{26})$ equivalent
to $(1,e^{3},1,e,e^{13})$ up to the Frobenius automorphism. The list
of Theorem \ref{thm:r=00003D4} for $r=4$ is actually complete and
non-repetitive up to compositions of linear transformations and Frobenius
automorphisms.
\end{rem*}

\subsection{Case $r=5$, $q=32$}

As $\mathbf{HC}(5,7,0,0,0,a_{4},a_{3},a_{2},a_{1})=a_{3}^{5}$, so
$a_{3}=0$ if $a_{7}=a_{6}=a_{5}=0$. As $(a_{7},a_{6},a_{5},a_{3})\neq(0,0,0,0)$,
$(a_{7},a_{6},a_{5})\neq(0,0,0)$. By $(\mathbf{R1})$, $(a_{7},a_{6})=(0,0)$,
$(1,0)$ or $(0,1)$.

(1) Suppose $a_{7}=a_{6}=0\neq a_{5}$. By $(\mathbf{R3})$, $a_{5}=1$
and $a_{4}=0$. We have $a_{1}=a_{3}^{5}+a_{3}^{2}$ as 
\[
0=\mathbf{HC}(5,7,0,0,1,0,a_{3},a_{2},a_{1})=a_{3}^{5}+a_{3}^{2}+a_{1}.
\]

(2) Suppose $(a_{7},a_{6})=(1,0)$. Then $\mathbf{HC}(5,5,1,0,a_{5},a_{4},a_{3},a_{2},a_{1})=a_{3}=0$.

(3) Suppose $(a_{7},a_{6})=(0,1)$. By $(\mathbf{R2})$, $a_{4}\in\{0,\omega(a_{5})\}$
if $a_{5}\neq0$; $a_{4}=0$ if $a_{5}=0$.

\begin{algorithm}[h]
\protect\caption{\label{alg:r=00003D5} Search for non-exceptional PPs of degree $8$
over $\mathbb{F}_{32}$}

\begin{lstlisting}
q = 2^5; F.<e> = GF(q,modulus=x^5+x^2+1,repr='log')
for (a3,a2) in [(a3,a2) for a3 in F for a2 in F]:### a7=a6==a4=0,a5=1 ###
    a1 = a3^5+a3^2
    if PP(q,e,0,0,1,0,a3,a2,a1):
        print(0,0,F(1),0,a3,a2,a1)
for a5 in F: ### a7=1,a6=a3=0 ###
    for (a4,a2,a1) in [(a4,a2,a1) for a4 in F for a2 in F for a1 in F]:
        if PP(q,e,1,0,a5,a4,0,a2,a1):
            print(F(1),0,a5,a4,0,a2,a1)
for a5 in F: ### a7=0,a6=1 ###
    if a5==0: A4 = [0]
    else:
        Image = [u^2+a5*u for u in F]
        for w in F:
            if w not in Image: break
        A4 = [0,w]
    for (a4,a3) in [(a4,a3) for a4 in A4 for a3 in F]:
        for (a2,a1) in [(a2,a1) for a2 in F for a1 in F]:
            if PP(q,e,0,1,a5,a4,a3,a2,a1):
                print(0,F(1),a5,a4,a3,a2,a1)
\end{lstlisting}
\end{algorithm}

Therefore, we write Algorithm \ref{alg:r=00003D5} to search for all
non-exceptional PPs of degree $8$ over $\mathbb{F}_{32}$ up to linear
transformations, and run it in SageMath 8.6.
\begin{itemize}
\item The outputs of Algorithm \ref{alg:r=00003D5} are $20$ tuples of
the form $(0,1,a_{5},a_{4},a_{3},a_{2},a_{1})$ such that $a_{5}(a_{3}+a_{5}^{3})\neq0$.
By Proposition \ref{prop:Req}, they are linearly related into exactly
$10$ pairs, which correspond to all $10$ linearly related classes
of non-exceptional PPs of degree $8$ over $\mathbb{F}_{32}$.
\item In the outputs of Algorithm \ref{alg:r=00003D5}, $a_{5}\in\{e,e^{2},e^{4},e^{8},e^{11},e^{13},e^{16},e^{21},e^{22},e^{26}\}$.
By Proposition \ref{prop:Frobenius}, up to a composition of the Frobenius
automorphism with itself, we only need to pick up outputting tuples
$(0,1,a_{5},a_{4},a_{3},a_{2},a_{1})$ such that $a_{5}\in\{e,e^{11}\}$.\end{itemize}
\begin{thm}
\label{thm:r=00003D5}Let $e$ be a root of the Conway polynomial
$x^{5}+x^{2}+1$ in $\mathbb{F}_{32}$. Each non-exceptional PP of
degree $8$ over $\mathbb{F}_{32}$ is linearly related to one of
the following: 
\begin{gather*}
x^{8}+x^{6}+e^{2^{j}}x^{5}+e^{26\cdot2^{j}}x^{3}+e^{25\cdot2^{j}}x^{2},\\
x^{8}+x^{6}+e^{11\cdot2^{j}}x^{5}+e^{2^{j}}x^{4}+e^{29\cdot2^{j}}x^{3}+e^{27\cdot2^{j}}x,
\end{gather*}
with $j$ running through the set $\{0,1,2,3,4\}$.
\end{thm}

\subsection{Case $r=6$, $q=64$}

First, $a_{7}=0$ as $\mathbf{HC}(6,9,a_{7},a_{6},a_{5},a_{4},a_{3},a_{2},a_{1})=a_{7}^{9}$.
Note that 
\begin{align*}
0 & =\mathbf{HC}(6,11,0,a_{6},a_{5},a_{4},a_{3},a_{2},a_{1})=a_{5}^{3}a_{6}^{8}+a_{3}a_{6}^{10}=a_{6}^{8}(a_{5}^{3}+a_{3}a_{6}^{2}),\\
0 & =\mathbf{HC}(6,21,0,a_{6},a_{5},a_{4},a_{3},a_{2},a_{1})=a_{3}^{21}+a_{6}^{21}.
\end{align*}
Note the relations: $a_{3}\neq0\Leftrightarrow a_{6}\neq0\Rightarrow a_{5}^{3}=a_{3}a_{6}^{2}\neq0$.
Recall that $(a_{6},a_{5},a_{3})\neq(0,0,0)$. If $a_{6}=0$, then
$a_{3}=0\neq a_{5}$. In any case, we have $a_{5}\neq0$. By $(\mathbf{R}1)$,
$(a_{7},a_{6})=(0,0)$, or $(0,1)$.

(1) Suppose $(a_{7},a_{6})=(0,0)$. Already $a_{3}=0\neq a_{5}$.
By $(\mathbf{R3})$, $a_{4}=0$ and $a_{5}\in\{1,e,e^{2}\}$.

(2) Suppose $(a_{7},a_{6})=(0,1)$. As mentioned above, $a_{3}=a_{5}^{3}\neq0$.
By $(\mathbf{R2})$, $a_{4}\in\{0,\omega(a_{5})\}$.

\begin{algorithm}[h]
\protect\caption{\label{alg:r=00003D6} Search for non-exceptional PPs of degree $8$
over $\mathbb{F}_{64}$}

\begin{lstlisting}
q = 2^6; F.<e> = GF(q,modulus=x^6+x^4+x^3+x+1,repr='log')
for a5 in [F(1),e,e^2]: ### a7=a6==a4=a3=0!=a5 ###
    for (a2,a1) in [(a2,a1) for a2 in F for a1 in F]:
        if PP(q,e,0,0,a5,0,0,a2,a1):
            print(0,0,a5,0,0,a2,a1)
for a5 in [a5 for a5 in F if a5!=0]: ### a7=0,a6=1,a5!=0 ###
    Image = [u^2+a5*u for u in F]
    for w in F:
        if w not in Image: break
    for a4 in [0,w]:
        a3 = a5^3
        for (a2,a1) in [(a2,a1) for a2 in F for a1 in F]:
            if PP(q,e,0,1,a5,a4,a3,a2,a1):
                print(0,F(1),a5,a4,a3,a2,a1)
\end{lstlisting}
\end{algorithm}

Therefore, we write Algorithm \ref{alg:r=00003D6} to search for all
non-exceptional PPs of degree $8$ over $\mathbb{F}_{64}$ up to linear
transformations, and run it in SageMath 8.6. The outputs can be reworded
as the following theorem.
\begin{thm}
\label{thm:r=00003D6}Let $e$ be a root of the Conway polynomial
$x^{6}+x^{4}+x^{3}+x+1$ in $\mathbb{F}_{64}$. Each non-exceptional
PP of degree $8$ over $\mathbb{F}_{64}$ is linearly related to one
of the following: 
\begin{gather*}
x^{8}+ex^{5}+e^{2}x^{2},\quad x^{8}+e^{2}x^{5}+e^{4}x^{2},\quad x^{8}+x^{6}+x^{5}+e^{3}x^{4}+x^{3}+e^{14}x^{2}+e^{6}x.
\end{gather*}

\end{thm}

\subsection{Case $r=7$, $q=128$}
\begin{thm}
All PPs of degree $8$ over $\mathbb{F}_{128}$ are exceptional.\end{thm}
\begin{proof}
Let us prove it by reduction to absurdity. Suppose that $f$ is a
non-exceptional PP of degree $8$ over $\mathbb{F}_{128}$. Without
loss of generality, we can assume that $f(x)=x^{8}+\sum_{i=1}^{7}a_{i}x^{i}$
with all $a_{i}\in\mathbb{F}_{128}$ satisfying requirements $(\mathbf{R1})$
and $(\mathbf{R2})$ of Proposition \ref{prop:Req}. As we mentioned
in the introduction section, a linearized PP must be exceptional.
In particular, we note that: \begin{itemize} 

\item $(a_{7},a_{6},a_{5},a_{3})\neq(0,0,0,0)$;

\item $(a_{7},a_{6})\in\{(1,0),(0,1),(0,0)\}$;

\item if $(a_{7},a_{6},a_{5})=(0,1,0)$ then $a_{4}=0$.

\end{itemize}

Note that
\begin{alignat*}{2}
 & \mathbf{HC}(7,23,0,0,a_{5},a_{4},a_{3},a_{2},a_{1}) &  & =a_{5}^{19},\\
 & \mathbf{HC}(7,29,0,0,0,a_{4},a_{3},a_{2},a_{1}) &  & =a_{3}^{21}.
\end{alignat*}
If $a_{7}=a_{6}=0$, then $a_{5}=a_{3}=0$. So $(a_{7},a_{6})\neq(0,0)$.
Thus $(a_{7},a_{6})=(0,1)$ or $(1,0)$.

(1) Suppose $(a_{7},a_{6},a_{5})=(0,1,0)$. Then $a_{4}=0$. Note
that 
\begin{alignat*}{2}
 & \mathbf{HC}(7,43,0,1,0,0,0,a_{2},a_{1}) &  & =a_{2},\\
 & \mathbf{HC}(7,55,0,1,0,0,0,0,a_{1}) &  & =a_{1}^{16}.
\end{alignat*}
If $a_{3}=0$, then $a_{2}=a_{1}=0$, and $f(x)=x^{8}+x^{6}$ is not
a PP. So $a_{3}\neq0$. Also note that 
\begin{align*}
0 & =\mathbf{HC}(7,23,0,1,0,0,a_{3},a_{2},a_{1})=a_{3}^{5}+a_{2}^{2}a_{3}+a_{1}a_{3}^{2},\\
0 & =\mathbf{HC}(7,27,0,1,0,0,a_{3},a_{2},a_{1})=a_{3}^{17}+a_{2}^{2}a_{3}^{9}+a_{1}a_{3}^{10}+a_{2}^{8}a_{3},\\
0 & =\mathbf{HC}(7,23,0,1,0,0,a_{3},a_{2},a_{1})\cdot a_{3}^{8}+\mathbf{HC}(7,27,0,1,0,0,a_{3},a_{2},a_{1})\\
 & =a_{3}^{17}+a_{3}^{13}+a_{2}^{8}a_{3}=a_{3}(a_{3}^{4}+a_{3}^{3}+a_{2}^{2})^{4}.
\end{align*}
Let $a_{3}=t^{2}$ with $t\in\mathbb{F}_{128}^{*}$. Then $a_{2}=t^{4}+t^{3}$,
$a_{1}=a_{3}^{3}+a_{2}^{2}a_{3}^{-1}=t^{4}$, and 
\[
f(x)=x^{8}+x^{6}+t^{2}x^{3}+(t^{4}+t^{3})x^{2}+t^{4}x.
\]
Note that $f(t+1)=t^{8}+t^{5}=f(t)$. So $f$ is not a PP over $\mathbb{F}_{128}$.

(2) Suppose $(a_{7},a_{6})=(0,1)$ and $a_{5}\neq0$. Note that 
\begin{align*}
0 & =\mathbf{HC}(7,23,0,1,a_{5},a_{4},a_{3},a_{2},a_{1})^{4}\text{\ensuremath{\cdot}}a_{3}+\mathbf{HC}(7,29,0,1,a_{5},a_{4},a_{3},a_{2},a_{1})\\
 & =a_{3}a_{5}^{76}+a_{3}^{5}a_{5}^{64}=a_{3}a_{5}^{64}(a_{5}^{3}+a_{3})^{4}.
\end{align*}
As $a_{5}\neq0$, we have $a_{3}\in\{0,a_{5}^{3}\}$. \begin{itemize} 

\item Suppose $a_{3}=0$. Note that
\begin{align*}
0 & =\mathbf{HC}(7,27,0,1,a_{5},a_{4},0,a_{2},a_{1})+\mathbf{HC}(7,43,0,1,a_{5},a_{4},0,a_{2},a_{1})^{8}\cdot a_{5}^{3}\\
 & =a_{2}^{256}a_{5}^{27}+a_{5}^{67}+a_{1}^{2}a_{5}^{25}=a_{2}^{2}a_{5}^{27}+a_{5}^{67}+a_{1}^{2}a_{5}^{25}.
\end{align*}
So $a_{1}=a_{5}^{21}+a_{2}a_{5}$. Also note that 
\begin{align*}
0 & =\mathbf{HC}(7,23,0,1,a_{5},a_{4},0,a_{2},a_{5}^{21}+a_{2}a_{5})\\
 & =a_{5}^{43}+a_{5}^{27}+a_{5}^{19}+a_{2}a_{5}^{7}+a_{4}^{4}a_{5}^{3}+a_{2}^{2}a_{5}^{3}.
\end{align*}
So $a_{4}=a_{2}^{64}+a_{2}^{32}a_{5}+a_{5}^{10}+a_{5}^{6}+a_{5}^{4}$.
Then  
\begin{align*}
0 & =\mathbf{HC}(7,27,0,1,a_{5},a_{2}^{64}+a_{2}^{32}a_{5}+a_{5}^{10}+a_{5}^{6}+a_{5}^{4},0,a_{2},a_{5}^{21}+a_{2}a_{5})\\
 & =a_{2}^{1024}a_{5}^{3}+a_{2}^{256}a_{5}^{27}+a_{5}^{163}+a_{5}^{67}+a_{5}^{51}+a_{2}^{2}a_{5}^{27}+a_{2}^{8}a_{5}^{3}\\
 & =a_{5}^{67}+a_{5}^{51}+a_{5}^{36}=a_{5}^{35}(a_{5}^{32}+a_{5}^{16}+a_{5}).
\end{align*}
and $0=(a_{5}^{32}+a_{5}^{16}+a_{5})^{8}=a_{5}^{2}+a_{5}+a_{5}^{8}$.
So $a_{5}$ is a root of the Conway polynomial $x^{7}+x+1$ in $\mathbb{F}_{128}$,
and a generator of the multiplicative group $\mathbb{F}_{128}^{*}$.
The following codes take an arbitrary root $e$ of $x^{7}+x+1$ as
the inputting value of $a_{5}$. 
\begin{lstlisting}
F.<e> = GF(2^7,modulus=x^7+x+1)
L.<a2> = PolynomialRing(F)
HC(7,31,0,1,e,a2^64+a2^32*e+e^10+e^6+e^4,0,a2,e^21+a2*e)%(a2^128-a2)
\end{lstlisting}
Here ``\texttt{\%}'' is the modulus operator, giving the reduction
of a polynomial (in $\mathbb{F}_{128}[a_{2}]$) modulo $a_{2}^{128}-a_{2}$.
However, the output is a non-zero constant $e^{2}+e$ in $\mathbb{F}_{128}$.
This makes a contradiction.

\item Suppose $a_{3}=a_{5}^{3}$. Note that 
\begin{align*}
0 & =\mathbf{HC}(7,23,0,1,a_{5},a_{4},a_{5}^{3},a_{2},a_{1})\\
 & =a_{5}^{15}+a_{5}^{11}+a_{4}^{2}a_{5}^{7}+a_{2}^{2}a_{5}^{3}+a_{1}^{2}a_{5}\\
 & =a_{5}(a_{5}^{7}+a_{5}^{5}+a_{4}a_{5}^{3}+a_{2}a_{5}+a_{1})^{2}.
\end{align*}
So $a_{1}=a_{5}^{7}+a_{5}^{5}+a_{4}a_{5}^{3}+a_{2}a_{5}$, and  
\[
f(x)=x^{8}+x^{6}+a_{5}x^{5}+a_{4}x^{4}+a_{5}^{3}x^{3}+a_{2}x^{2}+(a_{5}^{7}+a_{5}^{5}+a_{4}a_{5}^{3}+a_{2}a_{5})x.
\]
Note that $f(a_{5})=0=f(0)$ with $a_{5}\neq0$. So $f$ is not a
PP over $\mathbb{F}_{128}$. \end{itemize} 

(3) Suppose $(a_{7},a_{6})=(1,0)$. Note that 
\begin{align*}
0=\mathbf{HC}(7,19,1,0,a_{5},a_{4},a_{3},a_{2},a_{1}) & =a_{5}^{3}+a_{4}^{2}+a_{1},\\
0=\mathbf{HC}(7,37,1,0,a_{5},a_{4},a_{3},a_{2},a_{1}) & =a_{3}^{33}+a_{2}.
\end{align*}
So $a_{1}=a_{5}^{3}+a_{4}^{2}$ and $a_{2}=a_{3}^{33}$. Also note
that  
\[
0=\mathbf{HC}(7,27,1,0,a_{5},a_{4},a_{3},a_{3}^{33},a_{5}^{3}+a_{4}^{2})=a_{4}^{16}a_{5}^{8}+a_{3}^{8}a_{5}^{16}+a_{3}^{16}=(a_{4}^{2}a_{5}+a_{3}a_{5}^{2}+a_{3}^{2})^{8}.
\]
So $a_{4}^{2}a_{5}+a_{3}a_{5}^{2}+a_{3}^{2}=0$. Thus $a_{4}=a_{3}^{64}a_{5}^{64}+a_{3}a_{5}^{63}$,
and $a_{1}=a_{5}^{3}+a_{3}a_{5}+a_{3}^{2}a_{5}^{126}$. Then $(a_{3},a_{5})\in\mathbb{F}_{128}\times\mathbb{F}_{128}^{*}$
is a common solution of three equations: 
\[
\left\{ \begin{aligned}\mathbf{HC}(7,23,1,0,a_{5},a_{3}^{64}a_{5}^{64}+a_{3}a_{5}^{63},a_{3},a_{3}^{33},a_{5}^{3}+a_{3}a_{5}+a_{3}^{2}a_{5}^{126}) & =0,\\
\mathbf{HC}(7,29,1,0,a_{5},a_{3}^{64}a_{5}^{64}+a_{3}a_{5}^{63},a_{3},a_{3}^{33},a_{5}^{3}+a_{3}a_{5}+a_{3}^{2}a_{5}^{126}) & =0,\\
\mathbf{HC}(7,31,1,0,a_{5},a_{3}^{64}a_{5}^{64}+a_{3}a_{5}^{63},a_{3},a_{3}^{33},a_{5}^{3}+a_{3}a_{5}+a_{3}^{2}a_{5}^{126}) & =0,
\end{aligned}
\right.
\]
which can be solved by the following SageMath codes:

\begin{lstlisting}
F.<e> = GF(2^7,modulus=x^7+x+1)
M.<a3,a5> = PolynomialRing(GF(2))
h23 = HC(7,23,1,0,a5,a3^64*a5^64+a3*a5^63,a3,a3^33,a5^3+a3*a5+a3^2*a5^126)
h29 = HC(7,29,1,0,a5,a3^64*a5^64+a3*a5^63,a3,a3^33,a5^3+a3*a5+a3^2*a5^126)
h31 = HC(7,31,1,0,a5,a3^64*a5^64+a3*a5^63,a3,a3^33,a5^3+a3*a5+a3^2*a5^126)
for (b,c) in [(b,c) for b in F for c in F if c!=0]:
    if h23(b,c)==h29(b,c)==h31(b,c)==0:
        print(b,c)
\end{lstlisting}
Nothing is printed in the output, which means that these three equations
on $(a_{3},a_{5})$ have no common solution in $\mathbb{F}_{128}\times\mathbb{F}_{128}^{*}$.
This makes a contradiction.
\end{proof}

\subsection{Case $r=8$, $q=256$}
\begin{thm}
All PPs of degree $8$ over $\mathbb{F}_{256}$ are exceptional.\end{thm}
\begin{proof}
Let us prove it by reduction to absurdity. Suppose that $f$ is a
non-exceptional PP of degree $8$ over $\mathbb{F}_{256}$. Without
loss of generality, we can assume that $f(x)=x^{8}+\sum_{i=1}^{7}a_{i}x^{i}$
with all $a_{i}\in\mathbb{F}_{256}$ satisfying requirements $(\mathbf{R1})$
of Proposition \ref{prop:Req}. As we mentioned in the introduction
section, a linearized PP must be exceptional. In particular, we note
that: \begin{itemize} 

\item $(a_{7},a_{6},a_{5},a_{3})\neq(0,0,0,0)$;

\item $(a_{7},a_{6})\in\{(1,0),(0,1),(0,0)\}$.

\end{itemize}

Note that  
\begin{alignat*}{2}
 & \mathbf{HC}(8,51,0,0,a_{5},a_{4},a_{3},a_{2},a_{1}) &  & =a_{5}^{51},\\
 & \mathbf{HC}(8,55,0,0,0,a_{4},a_{3},a_{2},a_{1}) &  & =a_{3}^{37}.
\end{alignat*}
If $a_{7}=a_{6}=0$, then $a_{5}=a_{3}=0$. So $(a_{7},a_{6})\neq(0,0)$.
Thus $(a_{7},a_{6})=(0,1)$ or $(1,0)$.

(1) Suppose $(a_{7},a_{6})=(0,1)$. Note that 
\begin{align*}
0=\mathbf{HC}(8,43,0,1,a_{5},a_{4},a_{3},a_{2},a_{1}) & =a_{5}^{3}+a_{3},\\
0=\mathbf{HC}(8,85,0,1,a_{5},a_{4},a_{3},a_{2},a_{1}) & =a_{3}^{85}+1.
\end{align*}
So $a_{3}=a_{5}^{3}\neq0$. Also note that 
\begin{align*}
0 & =\mathbf{HC}(8,47,0,1,a_{5},a_{4},a_{5}^{3},a_{2},a_{1})^{4}\cdot a_{5}^{3}+\mathbf{HC}(8,61,0,1,a_{5},a_{4},a_{5}^{3},a_{2},a_{1})\\
 & =a_{5}^{191}+a_{5}^{175}+a_{4}^{8}a_{5}^{159}+a_{2}^{8}a_{5}^{143}+a_{1}^{8}a_{5}^{135}\\
 & =a_{5}^{135}(a_{5}^{7}+a_{5}^{5}+a_{4}a_{5}^{3}+a_{2}a_{5}+a_{1})^{8}.
\end{align*}
So $a_{1}=a_{5}^{7}+a_{5}^{5}+a_{4}a_{5}^{3}+a_{2}a_{5}$, and 
\[
f(x)=x^{8}+x^{6}+a_{5}x^{5}+a_{4}x^{4}+a_{5}^{3}x^{3}+a_{2}x^{2}+(a_{5}^{7}+a_{5}^{5}+a_{4}a_{5}^{3}+a_{2}a_{5})x.
\]
Note that $f(a_{5})=0=f(0)$ with $a_{5}\neq0$. So $f$ is not a
PP over $\mathbb{F}_{256}$.

(2) Suppose $(a_{7},a_{6})=(1,0)$. Then $\mathbf{HC}(8,37,1,0,a_{5},a_{4},a_{3},a_{2},a_{1})=a_{3}=0$.
Note that 
\begin{align*}
0=\mathbf{HC}(8,53,1,0,a_{5},a_{4},0,a_{2},a_{1}) & =a_{2}^{4}a_{5}^{48}+a_{2}^{4}a_{4}^{32}+a_{1}^{16}a_{2}^{4}\\
 & =a_{2}^{4}(a_{5}^{3}+a_{4}^{2}+a_{1})^{16}.
\end{align*}
Therefore, either $a_{2}=0$ or $a_{1}=a_{5}^{3}+a_{4}^{2}$. Also
note that 
\begin{align*}
0=\mathbf{HC}(8,43,1,0,a_{5},a_{4},0,a_{2},a_{1}) & =a_{2}^{8}a_{5}^{3}+a_{2}^{8}a_{4}^{2}+a_{1}a_{2}^{8}+a_{1}^{8}\\
 & =a_{2}^{8}(a_{5}^{3}+a_{4}^{2}+a_{1})+a_{1}^{8}=a_{1}^{8}.
\end{align*}
So $a_{1}=0$. Then  
\[
0=\mathbf{HC}(8,45,1,0,a_{5},a_{4},0,a_{2},0)=a_{5}^{32}+a_{2}^{12},
\]
and thus $a_{5}=(a_{5}^{32})^{8}=a_{2}^{96}$. Further, if $a_{2}=0$,
then $a_{5}=a_{2}^{96}=0$, and $a_{4}=0$ as 
\[
\mathbf{HC}(8,39,1,0,0,a_{4},0,0,0)=a_{4}^{6},
\]
 while $f(x)=x^{8}+x^{7}$ is not a PP. Therefore, $a_{2}\neq0$ in
this case. So $a_{4}^{2}=a_{5}^{3}=(a_{2}^{96})^{3}=a_{2}^{288}$,
and $a_{4}=a_{2}^{144}$. As 
\begin{align*}
0=\mathbf{HC}(8,39,1,0,a_{2}^{96},a_{2}^{144},0,a_{2},0) & =a_{2}^{386}+a_{2}^{4}.
\end{align*}
we have $a_{2}^{382}=1$, and $a_{2}=1$ as $\mathrm{gcd}(382,255)=1$.
Then $a_{5}=a_{4}=a_{2}=1$, $a_{3}=a_{1}=0$, and 
\[
\mathbf{HC}(8,55,1,0,1,1,0,1,0)=1\neq0\text{ in }\mathbb{F}_{256}.
\]
This makes a contradiction.
\end{proof}

\subsection{Case $r=9$, $q=512$}
\begin{thm}
All PPs of degree $8$ over $\mathbb{F}_{512}$ are exceptional.\end{thm}
\begin{proof}
Let us prove it by reduction to absurdity. Suppose that $f$ is a
non-exceptional PP of degree $8$ over $\mathbb{F}_{512}$. Without
loss of generality, we can assume that $f(x)=x^{8}+\sum_{i=1}^{7}a_{i}x^{i}$
with all $a_{i}\in\mathbb{F}_{512}$ satisfying requirements $(\mathbf{R1})$,
$(\mathbf{R2})$ and $(\mathbf{R3})$ of Proposition \ref{prop:Req}.
As we mentioned in the introduction section, a linearized PP must
be exceptional. In particular, we note that: \begin{itemize} 

\item $(a_{7},a_{6},a_{5},a_{3})\neq(0,0,0,0)$;

\item $(a_{7},a_{6})\in\{(1,0),(0,1),(0,0)\}$;

\item if $(a_{7},a_{6},a_{5})=(0,1,0)$, then $a_{4}=0$;

\item if $a_{7}=a_{6}=0\neq a_{5}$, then $(a_{5},a_{4})=(1,0)$.

\end{itemize}

As $\mathbf{HC}(9,73,a_{7},a_{6},a_{5},a_{4},a_{3},a_{2},a_{1})=a_{7}^{73}=0$,
$(a_{7},a_{6})=(0,0)$ or $(0,1)$. Note that 
\[
\mathbf{HC}(9,117,0,0,0,a_{4},a_{3},a_{2},a_{1})=a_{3}^{85}.
\]
If $a_{6}=a_{5}=0$, then $a_{3}=0$. Recall that $(a_{7},a_{6},a_{5},a_{3})\neq(0,0,0,0)$,
so $(a_{6},a_{5})\neq(0,0)$.

(1) Suppose $a_{7}=a_{6}=0\neq a_{5}$. Then $(a_{5},a_{4})=(1,0)$.
Note that 
\begin{alignat*}{2}
 & \mathbf{HC}(9,93,0,0,1,0,a_{3},a_{2},a_{1}) &  & =a_{3},\\
 & \mathbf{HC}(9,103,0,0,1,0,0,a_{2},a_{1}) &  & =a_{1},\\
 & \mathbf{HC}(9,107,0,0,1,0,0,a_{2},0) &  & =a_{2}^{8}.
\end{alignat*}
So $a_{3}=a_{1}=a_{2}=0$, and $f(x)=x^{8}+x^{5}$ is not a PP over
$\mathbb{F}_{512}$.

(2) Suppose $(a_{7},a_{6},a_{5})=(0,1,0)$. Then $a_{4}=0$. Note
that 
\begin{alignat*}{2}
 & \mathbf{HC}(9,171,0,1,0,0,0,a_{2},a_{1}) &  & =a_{2},\\
 & \mathbf{HC}(9,183,0,1,0,0,0,0,a_{1}) &  & =a_{1}^{16}.
\end{alignat*}
If $a_{3}=0$, then $a_{2}=a_{1}=0$ and $f(x)=x^{8}+x^{6}$ is not
a PP. So $a_{3}\neq0$. Note that 
\begin{align*}
0 & =\mathbf{HC}(9,87,0,1,0,0,a_{3},a_{2},a_{1})=a_{3}^{5}+a_{2}^{2}a_{3}+a_{1}a_{3}^{2},\\
0 & =\mathbf{HC}(9,91,0,1,0,0,a_{3},a_{2},a_{1})=a_{3}^{17}+a_{2}^{2}a_{3}^{9}+a_{1}a_{3}^{10}+a_{2}^{8}a_{3},\\
0 & =\mathbf{HC}(9,87,0,1,0,0,a_{3},a_{2},a_{1})\cdot a_{3}^{8}+\mathbf{HC}(9,91,0,1,0,0,a_{3},a_{2},a_{1})\\
 & =a_{3}^{17}+a_{3}^{13}+a_{2}^{8}a_{3}=a_{3}(a_{3}^{4}+a_{3}^{3}+a_{2}^{2})^{4}.
\end{align*}
Let $a_{3}=t^{2}$ with $t\in\mathbb{F}_{512}^{*}$. Then $a_{2}=t^{4}+t^{3}$,
$a_{1}=a_{3}^{3}+a_{2}^{2}a_{3}^{-1}=t^{4}$, and 
\[
f(x)=x^{8}+x^{6}+t^{2}x^{3}+(t^{4}+t^{3})x^{2}+t^{4}x.
\]
Note that $f(t+1)=t^{8}+t^{5}=f(t)$. So $f$ is not a PP over $\mathbb{F}_{512}$.

(3) Suppose $(a_{7},a_{6})=(0,1)$ and $a_{5}\neq0$. Note that 
\begin{align*}
0 & =\mathbf{HC}(9,87,0,1,a_{5},a_{4},a_{3},a_{2},a_{1})\cdot(a_{5}^{96}+a_{3}^{32})+\mathbf{HC}(9,103,0,1,a_{5},a_{4},a_{3},a_{2},a_{1})\\
 & =a_{5}^{115}+a_{3}a_{5}^{112}+a_{3}^{32}a_{5}^{19}+a_{3}^{33}a_{5}^{16}=a_{5}^{16}(a_{5}^{3}+a_{3})^{33}.
\end{align*}
So $a_{3}=a_{5}^{3}$. Also note that  
\begin{align*}
0 & =\mathbf{HC}(9,87,0,1,a_{5},a_{4},a_{5}^{3},a_{2},a_{1})\\
 & =a_{5}^{15}+a_{5}^{11}+a_{4}^{2}a_{5}^{7}+a_{2}^{2}a_{5}^{3}+a_{1}^{2}a_{5}\\
 & =a_{5}(a_{5}^{7}+a_{5}^{5}+a_{4}a_{5}^{3}+a_{2}a_{5}+a_{1})^{2}.
\end{align*}
So $a_{1}=a_{5}^{7}+a_{5}^{5}+a_{4}a_{5}^{3}+a_{2}a_{5}$, and 
\[
f(x)=x^{8}+x^{6}+a_{5}x^{5}+a_{4}x^{4}+a_{5}^{3}x^{3}+a_{2}x^{2}+(a_{5}^{7}+a_{5}^{5}+a_{4}a_{5}^{3}+a_{2}a_{5})x.
\]
Note that $f(a_{5})=0=f(0)$ with $a_{5}\neq0$. So $f$ is not a
PP over $\mathbb{F}_{512}$.
\end{proof}

\noindent \textbf{Acknowledgements.} This work was supported by the
Natural Science Foundation of Guangdong Province {[}No. 2018A030310080{]}.
The author was also sponsored by the National Natural Science Foundation
of China {[}No. 11801579{]}. Special thanks go to my lovely newborn
daughter, without whose birth should this paper have come out much
earlier.

\end{document}